\documentclass{article}
\usepackage[english]{babel}
\usepackage[utf8x]{inputenc}
\usepackage[T1]{fontenc}
\usepackage{caption}

\usepackage[a4paper,top=3cm,bottom=2cm,left=3cm,right=3cm,marginparwidth=1.75cm]{geometry}
\usepackage{subcaption}
\usepackage{graphicx}
\usepackage{siunitx}
     \usepackage{setspace}
\usepackage{multirow}
\usepackage{hhline}
\usepackage{mathrsfs} 
\usepackage{amsmath}
\usepackage{amssymb}
\usepackage{amsthm}
\usepackage{dsfont}
\usepackage{enumitem}   
\usepackage{enumitem}
\usepackage{multicol}
\usepackage{nomencl}
\makenomenclature

\renewcommand\nompreamble{\begin{multicols}{2}}
\renewcommand\nompostamble{\end{multicols}}

\usepackage{float}
\usepackage{graphicx,subcaption}
\usepackage{color}
\usepackage{tabu}
 
\usepackage{grffile}
\usepackage{amsmath}

\makeatletter

\newcolumntype{R}[1]{>{\raggedleft\arraybackslash}m{#1}} 
\newcolumntype{L}[1]{>{\raggedright\arraybackslash}m{#1}} 
\newcolumntype{C}[1]{>{\centering\arraybackslash}m{#1}}

\renewcommand{\fnum@figure}{\small\textbf{\figurename~\thefigure}} 
\makeatother
\usepackage[algoruled,vlined,boxed,longend,english,lined,boxed,commentsnumbered]{algorithm2e}
\numberwithin{equation}{section}%
\newtheorem{theorem}{Theorem}[section]

\theoremstyle{definition}
\newtheorem{definition}{Definition}[section]

\theoremstyle{remark}

\theoremstyle{remarks}

\renewcommand{\div}{\text{\div}}

\renewcommand{\div}{\text{div}}

{%

\begin{enumerate}}%
{\end{enumerate}
}

\def\be{\begin{equation}}
\def\ee{\end{equation}}
\def\ba{\begin{aligned}}
\def\ea{\end{aligned}}
\def\bes{\begin{equation*}}
\def\ees{\end{equation*}}
\def\bc{\begin{cases}}
\def\ec{\end{cases}}
\numberwithin{equation}{section}
\everymath{\displaystyle}

\usepackage{diagbox}

\title{\textbf{Bayesian Inference for a Time-Fractional HIV Model with Nonlinear Diffusion}}
\author{Mohamed BenSalah \thanks{ Mohamed BenSalah,\hfil\break
Department of Computer Science \newline
ISSAT of Sousse, University of Sousse, Rue Tahar Ben Achour, Sousse 4003, Tunisia,\hfil\break
E-Mail : mohamed.bensalah@fsm.rnu.tn} \, ,  Salih Tatar \thanks{ Salih Tatar, Corresponding Author \hfil\break
Department of Mathematics $\&$  Computer Science, College of Science and General Studies,\newline
Alfaisal University, Riyadh, KSA,\hfil\break
E-Mail : statar@alfaisal.edu} \, , S\"{u}leyman Ulusoy \thanks{ S\"{u}leyman Ulusoy,\hfil\break
Department of Mathematics and Physics, Faculty of Arts and Sciences,\newline
American  University of Ras Al Khaimah, Ras Al Khaimah, UAE\hfil\break
E-mail : suleyman.ulusoy@aurak.ac.ae}  } 

\begin{document}
\maketitle

\begin{abstract}
This study investigates an inverse problem associated with a time-fractional HIV infection model incorporating nonlinear diffusion. The model describes the dynamics of uninfected target cells, infected cells, and free virus particles, where the diffusion terms are nonlinear density functions. The primary objective is to recover the unknown diffusion functions by utilizing final-time measurement data. Due to the inherent ill-posedness of the inverse problem and the presence of measurement noise, we employ a Bayesian inference framework to obtain stable and reliable estimates while quantifying uncertainty. To solve the inverse problem efficiently, we develop an Iterative Regularizing Ensemble Kalman Method (IREKM), which enables the simultaneous estimation of multiple diffusion terms without requiring gradient information. Numerical experiments validate the effectiveness of the proposed method in reconstructing the unknown diffusion terms under different noise levels, demonstrating its robustness and accuracy. These findings contribute to a deeper understanding of HIV infection dynamics and provide a computational approach for parameter estimation in fractional diffusion models.
\end{abstract}
\maketitle


\noindent {{\bf AMS Subject Classifications:} 26A33, 35R11, 35R30, 62F15, 92D25 

\section{Introduction} \label{(sec1)}
In this paper, we study an inverse problem for the following time-fractional HIV infection model: 
\begin{eqnarray}\label{eq1}
\left \{ \begin{array}{l}
\partial_t ^\alpha u_1 - \div(D_1(u_1) \nabla u_1)  = - du_1 - \beta(u_1, u_3) - \sigma(u_1, u_2), (x, t) \in \Omega_T,  \\
\partial_t ^\alpha u_2 -   \div(D_2(u_2) \nabla u_2)  = \beta(u_1, u_3) + \sigma(u_1, u_2) - ru_2, \quad (x, t) \in \Omega_T,  \\
\partial_t ^\alpha u_3 -  \div(D_3(u_3) \nabla u_3)  = Nu_2 - eu_3, \qquad \qquad \qquad \qquad (x, t) \in \Omega_T,  \\
u_1(x,t) = u_2(x, t) = u_3(x, t) = 0,    ~ \qquad \qquad \qquad \qquad \qquad (x, t) \in \Gamma_T,\\
u_1(x,0) = u_{1,0}(x), u_2(x,0) = u_{2,0}(x), u_3(x,0) = u_{3,0}(x), \qquad   x \in \Omega,
\end{array} \right.
\end{eqnarray}
where $\partial_t ^\alpha$ is the Caputo fractional derivative of order $0 < \alpha <1$,   $\Omega \subset \mathbb{R}^n$ is a bounded domain with a boundary $\partial \Omega$, $T > 0$ is a final time,  $\Omega_T = \Omega \times (0,T)$, $\Gamma_T = \partial \Omega \times (0,T)$, $u_1$ is the concentration of healthy cells, $u_2$ is the concentration of infected cells, $u_3$ is the concentration of healthy cells virions, $D_1(u_1), D_2(u_2)$ and $D_3(u_3)$ are the nonlinear density functions for $u_1, u_2$ and $u_3$ respectively, $d = d(x,t) >0$ is the death rate for healthy cells, $N = N(x,t) > 0$ is the virus production rate, $e = e(x,t) > 0$ is the death rate of virions, $r = r(x,t) > 0$ is the death rate of infected cells, $\beta(u_1, u_3) = \beta_1 \,\frac{u_1u_3}{1+u_1}$, $\sigma(u_1, u_2) = \beta_2 \, \frac{u_1u_2}{1+u_1}$, $\beta_1 = \beta_1(x, t) > 0$ is the virus infection rate, $\beta_2 = \beta_2(x, t) > 0$ is the cell-cell infection rate. We assume that $u_{1,0}(x), u_{2,0}(x)$ and $u_{3,0}(x)$ are in $L_2(\Omega)$, the set of square-integrable functions in $\Omega$ and $d(x,t), \beta_1(x,t), \beta_2(x,t), r(x, t), N(x, t), e(x, t)$ are in $C(\Omega_T)$, the set of continuous functions on  $\Omega_T$. For the given inputs $\alpha$, $T$, $D_1(u_1), D_2(u_2), D_3(u_3)$, $d$, $\beta$, $\sigma$, $r$, $e$, $N$, $u_{1,0}$,  $u_{2,0}$, $u_{3,0}$, the problem (\ref{eq1}) is called the direct problem.  \\

HIV (human immunodeficiency virus) attacks immune cells, impairing the body's ability to combat infections and making individuals more susceptible to opportunistic diseases. If left untreated, HIV can progress to AIDS (acquired immunodeficiency syndrome). In recent years, several mathematical models have been developed to describe the dynamics of HIV infection and to analyze epidemic patterns associated with AIDS. For instance, in \cite{1}, a reaction-diffusion within-host HIV model incorporating cell mobility, spatial heterogeneity, and cell-to-cell transmission is proposed. The authors investigate the well-posedness and global dynamics of the model and analyze the influence of spatial heterogeneity and diffusion properties on virion propagation through numerical simulations. In \cite{2}, the authors introduce a mathematical model to examine the interaction between drug addiction and the spread of HIV/AIDS in Iranian prisons, studying the stability of the system. A separate study \cite{3} formulates an ordinary differential equation model describing the dynamics of HIV/AIDS, focusing on control strategies and determining optimal conditions for disease management. Furthermore, \cite{4} presents a preliminary analysis of HIV transmission dynamics, leveraging quantitative epidemiological data and simple deterministic models to describe viral spread and persistence. In \cite{5}, a mathematical model of HIV/AIDS transmission through sexual contact is proposed, incorporating asymptomatic and symptomatic phases as a system of differential equations. An epidemic model with multiple stages of infection and its global stability is analyzed in \cite{6}. Another study \cite{7} proposes an HIV dynamics model under distinct HAART (highly active antiretroviral therapy) regimens, predicting disease progression across all infection stages. However, fractional differential equations offer significant advantages over integer-order models in capturing the behavior of dynamical systems (see \cite{8} - \cite{17} and references therein). Based on this premise, the authors in \cite{18} modified the model from \cite{1} to study HIV spread and transmission dynamics within the host. The modifications include introducing time-fractional derivatives, density-dependent diffusion operators, and Holling type II functional responses to better capture HIV infection mechanisms within healthy cells. It is more realistic to consider density-dependent diffusion rather than linear diffusion in biological models, as the movement of infected cells through normal tissue resembles diffusion in a porous medium. Consequently, the random motility of infected cells is best represented as a function of the system’s state variables.\\

Recently, there has been growing interest in inverse problems involving fractional derivatives. In most of these studies, a fractional time derivative is considered, and the inverse problem typically involves determining this derivative, a source term, or a coefficient under some additional conditions. These problems are of significant physical and practical importance. However, the majority of existing research focuses on linear models (see \cite{19} - \cite{31} and the references therein). Inverse problems for nonlinear and nonlocal models constitute a relatively new area of study, with only a limited number of experimental investigations and analytical results. For instance, in \cite{nl1}, a nonlinear time-fractional inverse coefficient problem is examined, where the unknown coefficient depends on the gradient of the solution. The authors prove the existence of a unique solution to the direct problem and establish the existence of a quasi-solution for the inverse problem within the class of admissible coefficients. A similar problem, in which the unknown coefficient depends on the solution itself, is studied in \cite{nl2}. In \cite{nl3}, the authors analyze numerical solutions for both the direct and inverse problems of a nonlinear, nonlocal time-fractional equation. They propose a numerical method based on discretizing the minimization problem and employing the steepest descent method alongside a least-squares approach to solve the inverse problem. The direct and inverse problems for a nonlinear time-fractional diffusion equation are investigated in \cite{nl4}, where a quasi-Newton optimization method is applied to obtain the numerical solution of the inverse problem. To address the ill-posedness of the inverse problem, Tikhonov regularization is utilized. In \cite{nl5}, an iterative method of lines scheme is developed for the numerical solution of the time-fractional Richards equation with implicit Neumann boundary conditions. Additionally, \cite{nl6} examines an inverse problem for an inhomogeneous time-fractional diffusion equation in a one-dimensional real-positive semiaxis domain. The authors demonstrate that the inverse problem is severely ill-posed and apply a modified regularization method based on frequency-domain solutions to address this issue. In the present study, we investigate an inverse problem for the direct problem (\ref{eq1}). Unlike the aforementioned studies, we consider a nonlinear and nonlocal system with nonlinear functions on the side, where three unknown nonlinear density functions must be simultaneously determined as part of the inverse problem. This approach represents a novel contribution to the field of inverse problems. Our study can be regarded as a continuation of the series of works on fractional inverse problems discussed above.

This paper is organized as follows: In the next section, we introduce the direct and inverse problems, providing the necessary mathematical background and formulation.  Section \ref{bayes} presents the Bayesian inference framework and the Iterative Regularizing Ensemble Kalman Method (IREKM) used for solving the inverse problem. Section \ref{num} details numerical experiments that validate the proposed approach, including discussions on reconstruction accuracy, noise sensitivity, and computational performance. Finally, Section \ref{remarks} provides concluding remarks, summarizing key findings and potential directions for future research.

\section{The direct and the inverse problems} \label{(sec2)}

As mentioned in the Section \ref{(sec1)}, the fractional-time derivative $\partial_t^\alpha $ considered in (\ref{eq1}) is the Caputo fractional derivative of order $0 < \alpha < 1$ and is defined by
\begin{eqnarray}\label{eq2}
\partial_t^\alpha u(x,t) := \frac {1}{\Gamma (1-\alpha)} \int_0^t {(t-\xi)^{-\alpha}}   \frac{\partial u(x,\xi)}{\partial \xi } d \xi,
\end{eqnarray}
where  $\Gamma$ is the Gamma function.  Kilbas et al  \cite{Kil} and Podlubny  \cite{podbook} can be referred for properties of the Caputo fractional derivative.

\begin{definition} \cite{18} A set $ \mathbb{D}$ satisfying the following conditions is called the class of admissible coefficients :
\begin{enumerate}[label=(\roman*)]
\item The Carath\'eodory functions $D_i(s) : \Omega_T \times  \mathbb{R}^n \to \mathbb{R}, i = 1, 2, 3$, are continuous functions with respect to $s$. 
\item $\vert D_i(s) \xi \vert \leq \gamma_i \big [\wedge_i (x, t) + \vert \xi \vert \big] $, $\gamma_i > 0$, $\xi \in \mathbb{R}^n$, $\wedge_i (x, t) \in L_2(\Omega_T)$, $i = 1, 2, 3$.
\item $\big(D_i(s) \xi - D_i(\hat {s}) \hat {\xi} \big) \big(\xi - \hat {\xi}  \big) \geq  0$, $\xi, \hat {\xi}  \in \mathbb{R}^n$, $i = 1, 2, 3$.
\item $\vert D_i(s) \xi \vert . \, \xi  \geq b_i \vert \xi \vert ^2 $, $i = 1, 2, 3$.
\item $ D_i '(s) \leq L  $, $L > 0, i = 1, 2, 3$.
\end{enumerate}
\end{definition}
The conditions (ii) - (iv) are called Leray - Lions assumptions and arise in the solvability of the direct problem (\ref{eq1}), see \cite{18}, \cite{nl4}, \cite{nl5}, \cite{Boccardo}. Next, we define a weak solution to the direct problem (\ref{eq1}). 

\begin{definition} \cite{18},  A triplet of the functions $(u_1, u_2, u_3)$ is called a weak solution to the direct problem (\ref{eq1}) if $u_1, u_2, u_3 \in L^2\left( 0, T; H_0^1(\Omega) \right) \cap  W_2^\alpha \left( 0, T; L^2(\Omega) \right)$   such that the following integral identities hold for a.e. $t \in [0, T]$:
\begin{eqnarray}\label{eq3}
\left \{ \begin{array}{l}
 \int \limits_{\Omega} \partial_t ^\alpha u_1   \, \varphi_1 \, dx  +  \int \limits_{\Omega} D_1\left( u_1 \right)  \nabla u_1 \cdot \nabla \varphi_1 \, dx   = -\int \limits_{\Omega}  d \, u_1 \, \varphi_1 \, dx -\int \limits_{\Omega}  \beta(u_1, u_3) \, \varphi_1 \, dx  -\int \limits_{\Omega}  \sigma(u_1, u_2) \, \varphi_1 \, dx,\\
  \int \limits_{\Omega} \partial_t ^\alpha u_2   \, \varphi_2 \, dx  +  \int \limits_{\Omega} D_2\left( u_2 \right)  \nabla u_2 \cdot \nabla \varphi_2 \, dx   = \int \limits_{\Omega}  \beta(u_1, u_3) \, \varphi_2 \, dx  + \int \limits_{\Omega}  \sigma(u_1, u_2) \, \varphi_2 \, dx - \int \limits_{\Omega} r \, u_2 \, \varphi_2 \, dx,\\
  \int \limits_{\Omega} \partial_t ^\alpha u_3   \, \varphi_3 \, dx  +  \int \limits_{\Omega} D_3\left( u_3 \right)  \nabla u_3 \cdot \nabla \varphi_3 \, dx   = \int \limits_{\Omega} N\, u_2 \varphi_3 \, dx  - \int \limits_{\Omega} e \, u_3 \varphi_3 \, dx,
  \end{array} \right.
 \end{eqnarray}
for each $ \varphi_i  \in L^2\left( 0, T; H_0^1(\Omega) \right) \,   \cap \, W_2^\alpha \left( 0, T; L_2(\Omega) \right)$, i = 1, 2, 3, where  $$W_2^\alpha(0, T):=\bigg \{u \in L^2[0, T]:  \partial_t ^\alpha u \in L^2[0, T] \,  \bigg  \}$$  is the fractional Sobolev space of order $\alpha$.  Here we regard $u_i(x, t)$ as a mapping from $t \in (0, T)$ to $L^2(\Omega)$ and write $u_i(t) = u_i(\cdot, t)$.
\end{definition}
In \cite{18}, it is proved that the direct problem (\ref{eq1}) has a unique and non-negative weak solution. Now we prove the following theorem: 

\begin{theorem}\label{theorem1}

Let  $D_i, i = 1, 2, 3 \in \mathbb{D}$ and $(u_1, u_2, u_3) \in L^2\left( 0, T; H_0^1(\Omega) \right) \cap  W_2^\beta \left( 0, T; L^2(\Omega) \right)^3$ be the unique weak solution to the direct problem (\ref{eq1}). Then the following estimates holds: 

\begin{equation}\label{eq4}
\partial_t^\alpha  \bigg( \Vert u_1\Vert^2 + \Vert u_2\Vert^2 + \Vert u_3\Vert^2 \bigg) + c \, \bigg( \Vert u_1\Vert_{H_0^1(\Omega)}^2  + \Vert u_2\Vert_{H_0^1(\Omega)}^2 + \Vert u_3\Vert_{H_0^1(\Omega)}^2    \bigg)  \, \leq   C.
\end{equation}
\end{theorem}
\begin{proof}
By taking $\varphi_1 = u_1$, $\varphi_2 = u_2$ and $\varphi_3 = u_3$ in (\ref{eq3}), we have
\begin{eqnarray}\label{eq5}
\left \{ \begin{array}{l}
 \int \limits_{\Omega} \partial_t ^\alpha u_1   \, u_1 \, dx  +  \int \limits_{\Omega} D_1\left( u_1 \right)  \vert \nabla u_1 \vert ^2 \, dx   = -\int \limits_{\Omega}  d \, u_1 ^2  \, dx -\int \limits_{\Omega}  \beta(u_1, u_3) \, u_1 \, dx  -\int \limits_{\Omega}  \sigma(u_1, u_2) \, u_1 \, dx,\\
  \int \limits_{\Omega} \partial_t ^\alpha u_2   \, u_2 \, dx  +  \int \limits_{\Omega} D_2\left( u_2 \right) \vert \nabla u_2 \vert ^2 \, dx   = \int \limits_{\Omega}  \beta(u_1, u_3) \, u_2 \, dx  + \int \limits_{\Omega}  \sigma(u_1, u_2) \, u_2 \, dx - \int \limits_{\Omega} r  \, u_2^2 \, dx,\\
  \int \limits_{\Omega} \partial_t ^\alpha u_3   \, u_3 \, dx  +  \int \limits_{\Omega} D_3\left( u_3 \right) \vert \nabla u_3 \vert ^2 \, dx   = \int \limits_{\Omega} N\, u_2 u_3 \, dx  - \int \limits_{\Omega} e \, u_3^2 \, dx.
  \end{array} \right.
 \end{eqnarray}
 
We evaluate the first integral identity  in (\ref{eq5}) as follows
 \begin{eqnarray}\label{eq6}
  \begin{split}
  \frac{1}{2} \partial_t ^\alpha \Vert u_1 \Vert ^2 + b_1    &   \Vert \nabla u_1 \Vert ^2   \overset{\text{Alikhanov ineq. and (iv)}}{\leq}  \int \limits_{\Omega} \partial_t ^\alpha u_1   \, u_1 \, dx  +  \int \limits_{\Omega} D_1\left( u_1 \right)  \vert \nabla u_1 \vert ^2 \, dx  \\
& \qquad \qquad  \qquad \quad  =   \int \limits_{\Omega}  d \, u_1 ^2  \, dx -\int \limits_{\Omega}  \beta(u_1, u_3) \, u_1 \, dx  -\int \limits_{\Omega}  \sigma(u_1, u_2) \, u_1 \, dx\\
& \qquad \qquad  \quad  \overset{\text{Young's ineq.}}{\leq}  C_1 \bigg( \Vert u_1\Vert^2 + \Vert u_2\Vert^2 + \Vert u_3\Vert^2 \bigg),
  \end{split}
 \end{eqnarray} 
where we used that $\beta(u_1, u_3) \leq \beta_1(x, t) u_3 $, $\sigma(u_1, u_2)  \leq \beta_2(x, t) u_2$ and $\beta_1(x,t), \beta_2(x,t)$ are in $C(\Omega_T)$. Then (\ref{eq6}) implies that 
 \begin{eqnarray}\label{eq7}
  \frac{1}{2} \partial_t ^\alpha \Vert u_1 \Vert ^2 + b_1  \Vert \nabla u_1 \Vert ^2  \leq   C_1 \bigg( \Vert u_1\Vert^2 + \Vert u_2\Vert^2 + \Vert u_3\Vert^2 \bigg).
 \end{eqnarray} 
 Similarly we  get that
 \begin{eqnarray}\label{eq8}
   \begin{split}
 &  \frac{1}{2} \partial_t ^\alpha \Vert u_2 \Vert ^2 + b_2  \Vert \nabla u_2 \Vert ^2  \leq   C_2 \bigg( \Vert u_2\Vert^2 + \Vert u_3\Vert^2 \bigg),\\
  & \frac{1}{2} \partial_t ^\alpha \Vert u_3 \Vert ^2 + b_3  \Vert \nabla u_3 \Vert ^2  \leq   C_3 \bigg( \Vert u_2\Vert^2 + \Vert u_3\Vert^2 \bigg). 
    \end{split}
 \end{eqnarray} 
 Finally, combining (\ref{eq7}) and (\ref{eq8}) together, we deduce that
 \begin{eqnarray}\label{eq9}
\frac{1}{2} \partial_t ^\alpha \bigg( \Vert u_1 \Vert ^2 +  \Vert u_2 \Vert ^2 +  \Vert u_3 \Vert ^2  \bigg) +  b \bigg(  \Vert \nabla u_1 \Vert ^2 +  \Vert \nabla u_2 \Vert ^2  + \Vert \nabla u_3 \Vert ^2   \bigg) \leq c \bigg( \Vert u_1\Vert^2 + \Vert u_2\Vert^2  + \Vert u_3\Vert^2  \bigg). 
 \end{eqnarray} 
By Lemma 2.4 in \cite{frontier}, we obtain the proof. 
\end{proof}

%==============================================================================================================================================================================

\begin{theorem}\label{theorem2}
Suppose that a sequence of coefficients $\{D_{i,n}\} \in \mathbb{D}$ converges pointwise in $[0, \infty)$ to a function $D_i(u) \in \mathbb{D},  i = 1, 2, 3$. Then the sequence of solutions $u_{i,n} := u(x, t; D_{i,n})$ converges to the solution $u_i := u(x, t; D_i)  \in L^2\left( 0, T; H_0^1(\Omega) \right) \cap  W_2^\beta \left( 0, T; L^2(\Omega) \right) $, for $i = 1,2,3.$
\end{theorem}
\begin{proof} The proof can be derived by closely following Lemma 4 in \cite{18}. 
\end{proof}

%==============================================================================================================================================================================

%\begin{theorem}\label{theorem2}
%Suppose that a sequence of coefficients $D_m := \{D_{1m}, D_{2m}, D_{3m}\} \subset \mathbb{D}$ converges pointwise in $[0, \infty)$ to  $D := \{D_1, D_2, D_3\} \subset \mathbb{D}$. Then the sequence of solutions $u_m := u(x, t; D_m)$ converges to the solution $u := u(x, t; D)  \in L^2\left( 0, \mathcal{T}; H_0^1(\Omega) \right) \cap  W_2^\beta \left( 0, \mathcal{T}; L^2(\Omega) \right) $, where $u := u(x, t; D) $ denotes the solution of the direct problem (\ref{eq1}) for  $D \in \mathbb{D} $.
%\end{theorem}

In this paper we study an inverse problem that consists of determining the functions $(U(x,t), D(x,t))$, where $U(x,t) = (u_1(x,t), u_2(x,t), u_3(x,t))$ and $D(u) = \big(D_1(u), D_2(u), D_3(u)\big)$, from the following final time measured data:
\begin{equation}\label{eq17}
u_i (x, T) = \psi_i (x),   i = 1, 2, 3, (x, t) \in \Omega.
 \end{equation}
For the consistency of the initial conditions in (\ref{eq1}) and the measured data in (\ref{eq17}), we assume that $ \psi_i(x) = 0,  i = 1, 2, 3, x \in \partial \Omega $. We denote by $U(x,t ; D)$ the solution to the direct problem (\ref{eq1}) for a given $D  \in \mathbb{D} $ and set the following input-output mapping: 

\begin{equation}\label{eq21}
\mathbb{F}_i(D_i): D_i  \longrightarrow  u_i (x, T)= \psi_i (x) ,   i = 1, 2, 3, x \in \partial \Omega,
 \end{equation}
 where $\mathbb{F}_i(D_i):  \mathbb{D} \longrightarrow  H_0^1(\Omega).$ By Theorem \ref{theorem2} , it is clear that input-output mapping is continuous. The following theorem is straightforward with Theorem \ref{theorem2} and Theorems 1.1 and 1.2 in \cite{Trillos}. 

\begin{theorem}
If the prior $\mu_0$ is any measure with $\mu_0 {(X)} = 1$, then the Bayesian inverse problem of recovering $D(u) = \big(D_1(u), D_2(u), D_3(u)\big)$ from the data $ d = \mathbb{F}(D) + \varepsilon$, $ \varepsilon \thicksim \mathcal{N}(0, C)$,  is well-defined and the solution in the Bayesian framework depends continuously on the data.
\end{theorem}
\section{Bayesian framework}\label{bayes}
In this section, we present a numerical approach  to identifying the diffusion terms \( (D_1, D_2, D_3) \) in the context of the HIV infection model \eqref{eq1}. The available experimental data, denoted by \( y \), is assumed to correspond to key biological quantities: the concentrations of uninfected target cells \( u(x,t) \), infected cells \( v(x,t) \), and the free virus particles \( w(x,t) \), which are the solutions to the time-fractional HIV infection model. Specifically, the data points \( y = (u(x,T), v(x,T), w(x,T)) \) are measured at the final time \( T \), representing the observed spatial distribution of these biological quantities. These measurements provide critical insights into the dynamics of HIV infection, enabling the estimation of the diffusion terms that govern the spatial spread of the infection within tissue. Accurate identification of these terms is essential for understanding disease progression and optimizing treatment strategies.

Our objective is to estimate the diffusion terms \( (D_1, D_2, D_3) \) using the experimental data \( y \). However, the presence of measurement noise introduces significant uncertainty into the estimation process, complicating the inverse problem. Additionally, the inherent ill-posedness of the problem further exacerbates these challenges, necessitating the use of regularization techniques to obtain stable and reliable solutions.

To account for uncertainties arising from both observational noise and prior knowledge, we adopt a Bayesian inference framework. This approach not only estimates the unknown parameter vector \( D = (D_1, D_2, D_3) \) but also quantifies the associated uncertainty by constructing  the posterior probability distribution of the parameters. The theoretical foundations of this framework are well-established in the works of Chada et al. \cite{chada2018parameterizations}, Cotter et al. \cite{cotter2010approximation}, Dashti et al. \cite{dashti2011uncertainty}, Iglesias et al. \cite{iglesias2016regularizing, iglesias2015filter} , Kaipio et al. \cite{kaipio2006statistical} and Stuart et al. \cite{stuart2010inverse}. Furthermore, the numerical aspects of Bayesian inverse problems have been extensively studied by Yan et al. \cite{yan2017convergence}  - \cite{yan2019adaptive}.

To formalize the Bayesian framework, let us introduce \(\varepsilon\) as Gaussian noise with mean of zero and the covariance matrix \( C = \sigma^2 I \), where \( I \) is the identity matrix and \( \sigma \) is the noise standard deviation. The noisy observation data \( d \) can be expressed as:
\[
d = \mathbb{F}(D) + \varepsilon,
\]
where \( \mathbb{F}(D) \) represents the observation operator, which is defined by the experimental data \( y \).

In the Bayesian context, both \( D \) and \( d \) are treated as random variables, and the posterior distribution is given by Bayes' theorem:
\[
\mathbb{P}(D|d) \propto \mathbb{P}(d|D)\mathbb{P}(D),
\]
where \( \mathbb{P}(D) = \mathbb{P}(D_1) \otimes \mathbb{P}(D_2) \otimes \mathbb{P}(D_3) \) is the prior distribution before the data is observed, with each \( \mathbb{P}(D_i) \) assumed to be Gaussian: \( \mathbb{P}(D_1) = \mathcal{N}(m_1, C_1) \), \( \mathbb{P}(D_2) = \mathcal{N}(m_2, C_2) \), and \( \mathbb{P}(D_3) = \mathcal{N}(m_3, C_3) \). The likelihood function \( \mathbb{P}(d|D) \) is given by:
\[
\mathbb{P}(d|D) \propto \exp\left( -\frac{1}{2} \|C^{-\frac{1}{2}}(d - \mathbb{F}(D))\|^2 \right).
\]
Once the posterior distribution is obtained, the next step is to extract meaningful information, typically by sampling from the posterior density. Markov Chain Monte Carlo (MCMC) methods are commonly used for this purpose, but they are computationally expensive, especially for large-scale problems. Although MCMC offers an exhaustive exploration of the posterior distribution, its high computational cost can limit its practical use. In such cases, ensemble-based methods provide an efficient alternative, capturing the essential characteristics of the posterior distribution with fewer evaluations of the forward model.

In this study, we extend the Iterative Regularizing Ensemble Kalman Method (IREKM), as introduced by Iglesias et al. \cite{iglesias2015iterative} and Zhang et al. \cite{zhang2018bayesian}, to address the inverse problem of simultaneously determining the diffusion terms \( D = (D_1, D_2, D_3) \) in the time-fractional HIV infection model \eqref{eq1}.

Algorithm \ref{algo1} provides the pseudocode of the {\it IREKM}.

\vspace{0.2cm}
\begin{algorithm}[H]

\vspace{0.2cm}
\begin{enumerate}
\item[{\bf 1.}] Generate the initial ensemble $\left\{D_0^j\right\}_{j=1}^{N_e}$ from the prior distribution $\mathbb{P}(D)$ of $D=(D_1, D_2, D_3)$, where $N_e$ is the sample size. Let $d^j=d+\xi^j$, where $\xi^j \sim \mathcal{N}(0, C), j=1,2,3, \ldots, N_e$. Set $n=0.$
\item[{\bf 2.}] Let $D_n^j=(D_{1,n}^j, D_{2,n}^j, D_{3,n}^j)$, $j=1,2,\dots, N_e.$ Calculate
$z_{n}^j = \mathbb{F}(D_n^j),\quad \text{for}\quad j \in \{1, 2, \ldots, N_e\},$ and compute the ensemble mean $$\overline{z}_n = \frac{1}{N_e} \sum_{j=1}^{N_e}z_{n}^j.$$
Let 
    \begin{align*}
        C_{n}^{zz} &= \frac{1}{N_e - 1} \sum_{j=1}^{N_e} (\mathbb{F}(D_{n}^j) - \overline{z}_n)(\mathbb{F}(D_{n}^j) - \overline{z}_n)^T, 
         \\
         C_{n}^{D_i z} &= \frac{1}{N_e - 1} \sum_{j=1}^{N_e} (D_{i,n}^j - \overline{D}_{i,n})(\mathbb{F}(D_{n}^j) - \overline{z}_n)^T, \text{ for } i=1,2,3.
    \end{align*}
   where $$ \overline{D}_{i,n} = \frac{1}{N_e} \sum_{j=1}^{N_e} D_{i,n}.$$
\item[{\bf 3.}] Update each ensemble member
        \[D_{i,n+1}^j= D_{i,n}^j + C_{n}^{D_{i} z} (C_{n}^{zz} + \vartheta_n C)^{-1} (d^j - z_{n}^j), \quad j \in \{1, 2, \ldots, N_e\},\, i \in \{1,2,3\},\]
    where $\vartheta_n$ is chosen as follows:
Let $\vartheta_0$ be an initial guess, and $\vartheta_n^{i+1}=2^i \vartheta_0$. Choose $\vartheta_n=\vartheta_n^M$
where $M$ is the first integer such that
    \[ \vartheta_n^M \|C^{-\frac{1}{2}} (C_{n}^{zz} + \vartheta_n^M C)^{-1} (d - \overline{z}_n)\| \geq \mu \|C^{-1} (d - \overline{z}_n)\|, \]
    and $\mu \in (0, 1)$ is a constant.
\item[{\bf 4.}]  Increase $n$ by one and go Step 2, repeat the above procedure until a stopping criterion is satisfied.
\item[{\bf 5.}] Take $\overline{D}_n=(\overline{D}_{1,n}, \overline{D}_{2,n}, \overline{D}_{3,n})$ as the numerical solution of $(D_1, D_2, D_3)$.

\end{enumerate}
    \caption{\it Iterative Regularizing Ensemble Kalman Method ({\it IREKM})}
    \label{algo1}
\end{algorithm}
In Algorithm \ref{algo1}, we employ the Iterative Regularizing Ensemble Kalman Method (IREKM), a derivative-free optimization technique that effectively addresses the inverse problem of reconstructing the diffusion term in the HIV infection model. This method eliminates the need for computing adjoint problems, which are typically required in gradient-based optimization approaches. By bypassing the complexity of calculating gradients, particularly in cases where such computations are difficult or infeasible, IREKM provides a robust and efficient alternative for solving inverse problems in nonlinear models, such as those encountered in biological systems like HIV infection. Moreover, IREKM's ability to simultaneously recover multiple parameters without relying on traditional alternating iteration methods further enhances its suitability for reconstructing the diffusion term in this context. For a deeper understanding of the numerical theory behind IREKM, readers are referred to detailed discussions in works such as \cite{iglesias2015iterative, iglesias2013ensemble}. Schillings et al. \cite{schillings2018convergence} offers a convergence analysis for linear inverse problems, although the convergence of IREKM in nonlinear contexts remains an open research question, as highlighted by Iglesias et al. \cite{iglesias2015iterative}. This paper focuses on the practical implementation of IREKM for the reconstruction of the diffusion term within the HIV infection model \eqref{eq1}.

 \section{Numerical experiments}\label{num}
In this section, we present numerical experiments designed to validate the effectiveness of the proposed Bayesian inference framework for identifying the diffusion terms in the HIV infection model. These experiments demonstrate the capability of the Iterative Regularizing Ensemble Kalman Method (IREKM) to recover the parameters \( (D_1, D_2, D_3) \) from noisy observations, highlighting the robustness and accuracy. Through these simulations, we also examine the impact of different noise levels and prior information on the estimation process, providing valuable insights into the performance of the proposed approach.

\subsection{Numerical solutions of the forward problem}\label{sol-iter}
In this section, we aim to address the direct problem (\ref{eq1}) numerically. Specifically, we consider the direct problem (\ref{eq1}) with the inclusion of source functions to facilitate the numerical approximation while maintaining the essential characteristics of the original problem.
\begin{eqnarray}\label{general}
\left \{ \begin{array}{l}
\partial_t ^\alpha u_1 -  \div(D_1(u_1) \nabla u_1)  = - du_1 - \beta(u_1, u_3) - \sigma(u_1, u_2) + f_1(x,t), (x, t) \in \Omega_T,  \\
\partial_t ^\alpha u_2 -  \div(D_2(u_2) \nabla u_2)  = \beta(u_1, u_3) + \sigma(u_1, u_2) - ru_2 + f_2(x,t), \quad (x, t) \in \Omega_T,  \\
\partial_t ^\alpha u_3 -  \div(D_3(u_3) \nabla u_3)  = Nu_2 - eu_3 + f_3(x,t), \qquad \qquad \qquad \qquad (x, t) \in \Omega_T,  \\
u_1(x,t) = u_2(x, t) = u_3(x, t) = 0,      ~  ~  \qquad \qquad \qquad \qquad \qquad \qquad \qquad (x, t) \in \Gamma_T,\\
u_1(x,0) = u_{1,0}(x), u_2(x,0) = u_{2,0}(x), u_3(x,0) = u_{3,0}(x), \qquad \qquad \qquad   x \in \Omega. 
\end{array} \right.
\end{eqnarray}

In the following subsections, we outline our approach to numerically solving the coupled system. We start by describing the numerical techniques and methods used to approximate solutions to the system of equations. This includes an explanation of our discretization schemes and approximation strategies. In the second part, we perform rigorous tests and analyses to evaluate the accuracy and reliability of our numerical solutions. The validation process aims to confirm the robustness of our methodology in addressing the forward problem.
\subsubsection{Approximation method}
In this section, we explore the practical aspects of solving the forward problem within a numerical framework. Specifically, we employ the implicit-explicit finite difference scheme to approximate the behavior of the problem under consideration. Without loss of generality,  we restrict our analysis to the one-dimensional case, where the spatial domain $\Omega$ is defined as $[a, b]$. \\

For the temporal approximation, we adopt the approach outlined in \cite{lin2007finite} - \cite{salah5}. We  divide the time interval $[0, T]$ into $M$ subintervals using equidistant nodal points:
\[
t_0 = 0 < t_1 < \ldots < t_M = T,
\]
with $t_m = m h_t$ for $m=0, \ldots, M$, where $h_t = T/M$ is the step length. Based on Y. Lin and C. Xu \cite{lin2007finite}, we can approximate the Caputo fractional derivative in time using a simple quadrature formula as follows, for all $0 \leq m \leq M-1$:
\[
\begin{aligned}
\partial_t^\alpha \vartheta \left(x, t_{m+1}\right) & =\frac{1}{\Gamma(1-\alpha)} \sum_{j=0}^m \int_{t_j}^{t_{j+1}}  \frac{\partial_s \vartheta(x, s)}{\left(t_{m+1}-s\right)^\alpha}\,\mathrm{d} s \\
& =\frac{1}{\Gamma(1-\alpha)} \sum_{j=0}^m \frac{\vartheta \left(x, t_{j+1}\right)-\vartheta\left(x, t_j\right)}{h_t} \int_{t_j}^{t_{j+1}} \frac{\mathrm{d} s}{\left(t_{m+1}-s\right)^\alpha}+r_{h_t}^{m+1}\\
& =\frac{h_t^{-\alpha}}{\Gamma(2-\alpha)} \sum_{j=0}^m\big[ \vartheta \left(x, t_{j+1}\right)-\vartheta\left(x, t_j\right)\big]\, \big[ \left(m+1-j\right)^{1-\alpha}-\left(m-j\right)^{1-\alpha}\big]+r_{h_t}^{m+1}\\
& =\frac{h_t^{-\alpha}}{\Gamma(2-\alpha)} \sum_{j=0}^m\big[ \vartheta \left(x, t_{m+1-j}\right)-\vartheta\left(x, t_{m-j}\right)\big]\, \big[ \left(j+1\right)^{1-\alpha}-j^{1-\alpha}\big]+r_{h_t}^{m+1},
\end{aligned}
\]
where $r_{h_t}^{m+1}$ is the truncation error and satisfies $r_{h_t}^{m+1}\leq c\, h_t^{2-\alpha}$.

For the sake of simplification, let us introduce the notations $c_\alpha:=h_t^\alpha\, \Gamma(2-\alpha)$ and $b_j^\alpha:=\left(j+1\right)^{1-\alpha}-j^{1-\alpha}$, for $j=0,1,\dots,m$. We define the discrete fractional differential operator $\widetilde{\partial_t^\alpha}$ as:
\[
\widetilde{\partial_t^\alpha} \vartheta \left(x, t_{m+1}\right):=\frac{1}{c_\alpha} \sum_{j=0}^mb_j^\alpha\, \big[ \vartheta \left(x, t_{m+1-j}\right)-\vartheta\left(x, t_{m-j}\right)\big].
\]
It is straightforward to observe that $\widetilde{\partial_t^\alpha} \vartheta \left(x, t_{m+1}\right)$ can be rewritten as follows:
\[
\widetilde{\partial_t^\alpha} \vartheta \left(x, t_{m+1}\right):=\frac{1}{c_\alpha}\vartheta \left(x, t_{m+1}\right)-\frac{1}{c_\alpha}\Phi^m(\vartheta),
\]
where the term $\Phi^m(\vartheta)$ is defined as follows: 
\begin{equation*}
\Phi^m(\vartheta):=\left\{\begin{array}{ll}
\vartheta\left(x, t_{m}\right), & \text { if } m=0 \\
\vartheta\left(x, t_{m}\right)-\sum_{j=1}^mb_j^\alpha\, \big[ \vartheta \left(x, t_{m+1-j}\right)-\vartheta\left(x, t_{m-j}\right)\big], & \text { if } 1\leq m \leq M-1. 
\end{array}\right.
\end{equation*}
Hereafter, we denote by $\vartheta^m$ the approximated
solution at the time $t_m$, for all $m=0,\dots,M$, i.e. $\vartheta^m=\vartheta(x,t_m)$ for all $x \in \Omega$. Under these considerations, the time discretization of system \eqref{general} reads as follows:
\begin{equation}\label{approw-time}
\left\{\begin{array}{ll}
\widetilde{\partial_t^\alpha} u_1^{m+1}-\operatorname{\div}\left(D_1(u_1^{m}) \nabla u_1^{m+1}\right)=-d^{m+1} u_1^{m+1}-\beta(u_1^{m}, u_3^{m})-\sigma(u_1^{m}, u_2^{m})+f_{1}^{m+1}, & \text { in } \Omega \\
\widetilde{\partial_t^\alpha} u_2^{m+1}-\operatorname{\div}\left(D_2(u_2^{m}) \nabla u_2^{m+1}\right)=\beta(u_1^{m}, u_3^{m})+\sigma(u_1^{m}, u_2^{m})-r^{m+1} u_2^{m+1}+f_{2}^{m+1}, & \text { in } \Omega \\
\widetilde{\partial_t^\alpha} u_3^{m+1}-\operatorname{\div}\left(D_3(u_3^{m}) \nabla u_3^{m+1}\right) =N^{m+1} u_2^{m}-e^{m+1} u_3^{m+1}+f_{3}^{m+1}, & \text { in } \Omega.
\end{array}\right.
\end{equation}
for all $0\leq m \leq M-1$, with 
$$
\begin{aligned}
& u_1^0=u_{1,0}(x);\, u_2^0=u_{2,0}(x);\, u_3^0=u_{3,0}(x) \text { in } \Omega \\
& u_1^{m+1}=u_2^{m+1}=u_3^{m+1}=0 \text { on } \partial \Omega,\; \text { for all } 0\leq m \leq M-1.
\end{aligned}
$$

\noindent Similarly, for discretizing the spatial domain, we introduce a set of evenly spaced grid points $\left(x_i\right){0 \leq i \leq N}$ defined as $x_i=a+i h_x$, where $N$ is an integer, and the spacing $h_x$ is determined by $h_x=(b-a)/N$. At the boundaries of the domain $\Omega=[a,b]$, we have $x_0=a$ and $x_{N}=b$. For each pair $(x_i,t_m)$, with $0\leq m\leq M$ and $0\leq i\leq N$, we seek the numerical value of the solution, denoted as $\vartheta\left(x_i,t_m\right)$. We enforce boundary and initial conditions such that $\vartheta\left(x_0,t_m\right)=\vartheta\left(x_N,t_m\right)=0$ for all $0\leq m\leq M$ and $\vartheta\left(x_i,t_0\right)=0$ for all $0\leq i\leq N$. To approximate the spatial derivatives, we commonly employ central difference formulas due to their ability to provide enhanced accuracy. The differential equation is applied exclusively at the grid points, and the approximations for the first and second derivatives are as follows:
$$
\begin{gathered}
\partial_x \vartheta(x_i,t_m)\approx \frac{\vartheta(x_{i+1},t_m)-\vartheta(x_{i-1},t_m)}{2 h_x}, \\
\partial_x^2 \vartheta(x_i,t_m)\approx\frac{\vartheta(x_{i+1},t_m)-2\vartheta(x_{i},t_m)+\vartheta(x_{i-1},t_m)}{h_x^2}
\end{gathered}
$$
From now on, we denote by \( u_{k,i}^m \) the discrete mesh function approximating \( u_k(x_i,t_m) \) for \( k=1,2,3 \), \( i=0,\dots,N \), and \( m=0,\dots,M \). With this notation in place, we introduce a fully discrete scheme for the problem \eqref{general}:
\begin{itemize}
    \item Given
    $$u_{1,i}^0=u_{1,0}(x_i),\, \; u_{2,i}^0=u_{2,0}(x_i),\,\;  u_{3,i}^0=u_{3,0}(x_i), \; \; \text{ for all } 0 \leq i \leq N,$$
    and 
$$u_{1,i}^{m+1}=u_{2,i}^{m+1}=u_{3,i}^{m+1}=0 \text { for all } 0 \leq m \leq M-1\; \text { and } i \in \{0, N\}.$$
    \item Find the vector $\Lambda^{m+1}$, for all $0 \leq m \leq M-1$, which is defined as:
$$\Lambda^{m+1}:=\left(u_{1,1}^{m+1},\dots,u_{1,N-1}^{m+1}, u_{2,1}^{m+1},\dots,u_{2,N-1}^{m+1}, u_{3,1}^{m+1},\dots,u_{3,N-1}^{m+1} \right)^T$$ such that it satisfies the following linear system of equations
\begin{equation}\label{approximated}\small
\left\{\begin{array}{l}
\frac{c_\alpha}{h_x^2}({}_1\mathbb{D}_{i}^m-\frac{{}_1\mathbb{D}_{i+1}^m-{}_1\mathbb{D}_{i-1}^m}{4})u_{1,i-1}^{m+1}+(1+\frac{2{}_1\mathbb{D}_{i}^mc_\alpha}{h_x^2}+c_\alpha d_i^{m+1})u_{1,i}^{m+1}-\frac{c_\alpha}{h_x^2}({}_1\mathbb{D}_{i}^m+\frac{{}_1\mathbb{D}_{i+1}^m-{}_1\mathbb{D}_{i-1}^m}{4})u_{1,i+1}^{m+1}={}_1\mathcal{F}_i^m,  \\
\\
\frac{c_\alpha}{h_x^2}({}_2\mathbb{D}_{i}^m-\frac{{}_2\mathbb{D}_{i+1}^m-{}_2\mathbb{D}_{i-1}^m}{4})u_{2,i-1}^{m+1}+(1+\frac{2{}_2\mathbb{D}_{i}^mc_\alpha}{h_x^2}+c_\alpha r_i^{m+1})u_{2,i}^{m+1}-\frac{c_\alpha}{h_x^2}({}_2\mathbb{D}_{i}^m+\frac{{}_2\mathbb{D}_{i+1}^m-{}_2\mathbb{D}_{i-1}^m}{4})u_{2,i+1}^{m+1}={}_2\mathcal{F}_i^m,  \\
 \\
\frac{c_\alpha}{h_x^2}({}_3\mathbb{D}_{i}^m-\frac{{}_3\mathbb{D}_{i+1}^m-{}_3\mathbb{D}_{i-1}^m}{4})u_{3,i-1}^{m+1}+(1+\frac{2{}_3\mathbb{D}_{i}^mc_\alpha}{h_x^2}+c_\alpha e_i^{m+1})u_{3,i}^{m+1}-\frac{c_\alpha}{h_x^2}({}_3\mathbb{D}_{i}^m+\frac{{}_3\mathbb{D}_{i+1}^m-{}_3\mathbb{D}_{i-1}^m}{4})u_{3,i+1}^{m+1}={}_3\mathcal{F}_i^m, 
\end{array}\right.
\end{equation}
where
\begin{equation}\label{terms}
\left\{\begin{array}{l}
{}_1\mathbb{D}(x,t)=D_1(u_1(x,t)) ;\, {}_2\mathbb{D}(x,t)=D_2(u_2(x,t)); \, {}_3\mathbb{D}(x,t)=D_3(u_3(x,t)), \, \, \text{ in }  \Omega_T,  \\
\\
{}_1\mathcal{F}^m:=\Phi^m(u_1)-c_\alpha \beta(u_1^{m}, u_3^{m})-c_\alpha \sigma(u_1^{m}, u_2^{m})+c_\alpha f_1^{m+1}, \, \, \text{ in }  \Omega, \\
 \\
{}_2\mathcal{F}^m:=\Phi^m(u_2)+c_\alpha \beta(u_1^{m}, u_3^{m})+c_\alpha \sigma(u_1^{m}, u_2^{m})+c_\alpha f_2^{m+1}, \, \, \text{ in }  \Omega,\\
\\
{}_3\mathcal{F}^m:=\Phi^m(u_3)+c_\alpha N^{m+1}u_2^m+c_\alpha f_3^{m+1}, \, \, \text{ in }  \Omega.
\end{array}\right.
\end{equation}
\end{itemize}
Henceforth, it can be observed that this problem can be transformed into a matrix-based formulation to streamline its resolution. Specifically, we seek solutions to the system of equations:

\[
\mathcal{A}\, \Lambda^{m+1} = \mathcal{F},\, \, \text{ for all } 0\leq m \leq M-1,
\]

where the matrix $\mathcal{A}$ and the right-hand side $\mathcal{F}$ are defined as follows:

\[
\mathcal{A}=\left[\begin{array}{lll}
U & 0 & 0 \\
0 & V & 0 \\
0 & 0 & W
\end{array}\right]\; \; \text{ and }\; \; \mathcal{F}=\left[\begin{array}{c}
^{}_1F \\
^{}_2F  \\
^{}_3F
\end{array}\right].
\]
Within this context, the tridiagonal matrices 
\[
U \in \mathcal{M}_{N-1}(\mathbb{R}),\; V \in \mathcal{M}_{N-1}(\mathbb{R}),\; \text{ and } W \in \mathcal{M}_{N-1}(\mathbb{R})
\]
are defined for all $1\leq p \leq N-1$ as follows:
$$U_{p-1,p}=\frac{c_\alpha}{h_x^2}\bigg({}_1\mathbb{D}_{p}^m-\frac{{}_1\mathbb{D}_{p+1}^m-{}_1\mathbb{D}_{p-1}^m}{4}\bigg),\; U_{p,p}=1+\frac{2{}_1\mathbb{D}_{i}^mc_\alpha}{h_x^2}+c_\alpha d_i^{m+1},\; \text{ and } \, U_{p,p+1}=\frac{c_\alpha}{h_x^2}\bigg({}_1\mathbb{D}_{p}^m+\frac{{}_1\mathbb{D}_{p+1}^m-{}_1\mathbb{D}_{p-1}^m}{4}\bigg),$$
$$V_{p-1,p}=\frac{c_\alpha}{h_x^2}\bigg({}_2\mathbb{D}_{p}^m-\frac{{}_2\mathbb{D}_{p+1}^m-{}_2\mathbb{D}_{p-1}^m}{4}\bigg),\; V_{p,p}=1+\frac{2{}_2\mathbb{D}_{p}^mc_\alpha}{h_x^2}+c_\alpha r_p^{m+1},\; \text{ and } \, V_{p,p+1}=\frac{c_\alpha}{h_x^2}\bigg({}_2\mathbb{D}_{p}^m+\frac{{}_2\mathbb{D}_{p+1}^m-{}_2\mathbb{D}_{p-1}^m}{4}\bigg),$$
$$W_{p-1,p}=\frac{c_\alpha}{h_x^2}\bigg({}_3\mathbb{D}_{p}^m-\frac{{}_3\mathbb{D}_{p+1}^m-{}_3\mathbb{D}_{p-1}^m}{4}\bigg),\; W_{p,p}=1+\frac{2{}_3\mathbb{D}_{i}^mc_\alpha}{h_x^2}+c_\alpha e_i^{m+1},\; \text{ and } \, W_{p,p+1}=\frac{c_\alpha}{h_x^2}\bigg({}_3\mathbb{D}_{p}^m+\frac{{}_3\mathbb{D}_{p+1}^m-{}_3\mathbb{D}_{p-1}^m}{4}\bigg),$$
respectively. Furthermore, the vectors 
\[
^{}_1F \in \mathbb{R}^{N-1},\; ^{}_2F \in \mathbb{R}^{N-1},\; \text{ and }\, ^{}_3F \in \mathbb{R}^{N-1}
\]
are defined for all $1\leq p \leq N-1$ as follows:
\[
^{}_1F_p=^{}_1\mathcal{F}_p^m,\; ^{}_2F_p=^{}_2\mathcal{F}_p^m,\; \text{ and }\, ^{}_3F_p=^{}_3\mathcal{F}_p^m.
\]
To this end, the unknown term $\Lambda^{m+1}$, for $0\leq m \leq M-1$, can be obtained as follows: 
$$\Lambda^{m+1}=\mathcal{A}^{-1}\mathcal{F}.$$
\subsubsection{Validation test}
In this section, we conduct a validation test to assess the performance and accuracy of our numerical method. The objective is to compare the results obtained from our numerical approach with the exact solutions of the continuous model. To do so, we take the  computational domain as $\Omega=[0,1]$ and set the final time as $T=1$. We consider three exact solutions as follows:
$$u_{1}^\dagger(x,t)=t^\alpha \sin(\pi x),\, \, u_{2}^\dagger(x,t)=t x^2 (x-1), \, \text{ and }\, u_{3}^\dagger(x,t)=t^{2\alpha} x (1-x)\, \,  \text{ for all } (x,t) \in \Omega_T.$$
It is evident that the exact solutions satisfy the following initial and boundary conditions:
$$
\begin{aligned}
&u_{1}^\dagger(x, 0)= u_{2}^\dagger(x, 0)= u_{3}^\dagger(x, 0)=0 \text { in } \Omega \\
& u_{1}^\dagger(x, t)=u_{2}^\dagger(x, t)=u_{3}^\dagger(x, t)=0 \text { on } \Gamma_T.
\end{aligned}
$$
The diffusion terms $D_1$, $D_2$, and $D_3$ are chosen as follows:
$$D_1(u_1)=1+u_1,\; \; D_2(u_2)=(1+u_2)^2,\; \text{ and }\; D_3(u_3)=1+u_3^2.$$
The functions $d$, $\beta_1$, $\beta_2$, $r$, $N$, and $e$ are given by:
$$d(x,t)=1+tx,\; \; \beta_1(x,t)=0.1,\; \; \beta_2(x,t)=0.01,\; \; r(x,t)=0.3,\; \; e(x,t)=0.001,\; \text{ and }\, N(x,t)=e^{xt}.$$
The source terms $f_1$, $f_2$, and $f_3$ can be computed by substituting this exact solutions into the system \eqref{general}. To assess the accuracy of the proposed approximation method, we define the following error function:  
\[
\operatorname{\bf Err}_k(h_t,h_x) := \| u_{k}^\dagger - \tilde{u}_{k} \|_{L^\infty(0,T; L_2(\Omega))}, \quad \text{for } k=1,2,3.
\]  
Here, \( \tilde{u}_{k} \) represents the approximated solution corresponding to the exact solution \( u_{k}^\dagger \). The spatial and temporal convergence orders are computed by 
$$
\operatorname{\bf Ord}_x=\log _2 \frac{\operatorname{\bf Err}_k(h_t, h_x)}{\operatorname{\bf Err}_k(h_t, h_x/ 2)},\; \; \;\operatorname{\bf Ord}_t=\log _2 \frac{\operatorname{\bf Err}_k(h_t, h_x)}{\operatorname{\bf Err}_k(h_t/2, h_x)}.
$$
In Table \ref{order_x}, we present the calculated error values in relation to the variation in the mesh size $h_x$, along with the corresponding convergence order $\operatorname{\bf Ord}_x$. This analysis is carried out for two values of $\alpha$, namely $\alpha=0.3$ and $\alpha=0.7$ while keeping the step time size $h_t$ fixed at $h_t=0.001$.
\begin{table}[H]
    \centering
    \begin{tabular}{ c|| c c c c c c c  } 
\hline
 Derivative order $\alpha$& $h_x$ & $\operatorname{\bf Err}_1$ & $\operatorname{\bf Ord}_x$ & $\operatorname{\bf Err}_2$& $\operatorname{\bf Ord}_x$ & $\operatorname{\bf Err}_3$ & $\operatorname{\bf Ord}_x$ \\
\hline
\multirow{5}{4em}{$\alpha=0.3$} &$0.1$ & $1.97\text{\bf e-}2$& {\bf \Large-} & $4.49\text{\bf e-}3$& {\bf \Large-} & $6.53\text{\bf e-}4$ & {\bf \Large-} \\ 
&$0.05$ & $5.11\text{\bf e-}3$& $1.9468$ & $1.02\text{\bf e-}3$& $2.1381$ & $1.65\text{\bf e-}4$ & $1.9846$ \\ 
&$0.025$ & $1.21\text{\bf e-}3$& $2.0783$ & $2.61\text{\bf e-}4$& $1.9664$ & $4.23\text{\bf e-}5$ & $1.9637$ \\ 
&$0.0125$ & $3.33\text{\bf e-}4$& $1.8614$ & $6.72\text{\bf e-}5$& $1.9575$ & $1.04\text{\bf e-}5$ & $2.0240$ \\
& $0.00625$ & $7.45\text{\bf e-}5$ & $2.1602$ & $1.69\text{\bf e-}5$& $1.9914$ & $2.74\text{\bf e-}6$&$1.9243$ \\
\hline
\multirow{5}{4em}{$\alpha=0.7$} &$0.1$ & $2.21\text{\bf e-}2$& {\bf \Large-} & $4.77\text{\bf e-}3$& {\bf \Large-} & $6.58\text{\bf e-}4$ & {\bf \Large-} \\ 
&$0.05$ & $5.31\text{\bf e-}3$& $2.0572$ & $1.12\text{\bf e-}3$& $2.0904$ & $1.71\text{\bf e-}4$ & $1.9440$ \\ 
&$0.025$ & $1.38\text{\bf e-}3$& $1.9440$ & $2.70\text{\bf e-}4$& $2.0524$ & $4.29\text{\bf e-}5$ & $1.9949$\\ 
&$0.0125$ & $3.39\text{\bf e-}4$& $2.0295$ & $6.80\text{\bf e-}5$& $1.9893$ & $1.11\text{\bf e-}5$ & $1.9504$ \\
& $0.00625$ & $7.52\text{\bf e-}5$ & $2.1724$ & $1.81\text{\bf e-}5$& $1.9095$ & $2.64\text{\bf e-}6$&$2.0719$  \\
\hline
\end{tabular}
    \caption{Convergence analysis for different mesh sizes $h_x$ and two values of $\alpha$ ($\alpha=0.3$ and $\alpha=0.7$) with a fixed step time size $h_t=0.001$.}
    \label{order_x}
\end{table}

%\begin{figure}[H]
  %   \centering
   %  \begin{subfigure}[b]{0.45\textwidth}
         %\centering
         %\includegraphics[width=\textwidth]{FIGURES/space-0.3.png}
         %\caption {For $\alpha=0.3$}
         %\label{fig:y equals x}
    % \end{subfigure}
   %  \hfill
     %\begin{subfigure}[b]{0.45\textwidth}
        % \centering
         %\includegraphics[width=\textwidth]{FIGURES/space-0.7.png}
         %\caption{For $\alpha=0.7$}
         %\label{fig:three sin x}
     %\end{subfigure}   
        %\caption{Comparison of Exact and Approximated Solutions for $\alpha=0.5$ with $h_x=h_t=0.001$. The top row illustrates the exact solutions, while the bottom row presents the approximated solutions. }
       % \label{orderx}
%\end{figure}
%%%%%%%%%%%%%%%%%%%

In Table \ref{order_t}, we present the convergence analysis with respect to the step time size $h_t$ for $\alpha=0.3$ and $\alpha=0.7$ while keeping the mesh size fixed at $h_x=0.001$. The table provides the calculated error values and the corresponding convergence orders $\operatorname{\bf Ord}_t$ for different  time step sizes.

\begin{table}[H]
    \centering
    \begin{tabular}{ c|| c c c c c c c  } 
\hline
  Derivative order $\alpha$& $h_t$ & $\mathrm{\bf Err}_1$ & $\operatorname{\bf Ord}_t$ & $\operatorname{\bf Err}_2$& $\operatorname{\bf Ord}_t$ & $\operatorname{\bf Err}_3$ & $\operatorname{\bf Ord}_t$ \\
\hline
\multirow{5}{4em}{$\alpha=0.3$} &$0.1$ & $5.16\text{\bf e-}2$& {\bf \Large-} & $1.38\text{\bf e-}3$& {\bf \Large-} & $1.90\text{\bf e-}3$ & {\bf \Large-} \\ 
&$0.05$ & $2.73\text{\bf e-}2$& $0.9184$ & $7.15\text{\bf e-}4$& $0.9486$ & $9.75\text{\bf e-}4$ & $0.9625$ \\ 
&$0.025$ & $1.42\text{\bf e-}2$& $0.9430$ & $3.68\text{\bf e-}4$& $0.9582$ & $5.04\text{\bf e-}4$ & $0.9519$ \\ 
&$0.0125$ & $6.26\text{\bf e-}3$& $1.1816$ & $1.75\text{\bf e-}4$& $1.0723$ & $2.63\text{\bf e-}4$ & $0.9383$ \\
& $0.00625$ & $3.16\text{\bf e-}3$ & $0.9862$ & $8.11\text{\bf e-}5$& $1.1095$ & $1.27\text{\bf e-}4$&$1.0502$ \\
\hline
\multirow{5}{4em}{$\alpha=0.7$} &$0.1$ & $1.05\text{\bf e-}2$& {\bf \Large-} & $1.48\text{\bf e-}3$& {\bf \Large-} & $2.41\text{\bf e-}3$ & {\bf \Large-} \\ 
&$0.05$ & $5.24\text{\bf e-}3$& $1.0027$ & $7.47\text{\bf e-}4$& $0.9864$ & $1.21\text{\bf e-}3$ & $0.9940$ \\ 
&$0.025$ & $2.32\text{\bf e-}3$& $1.1754$ & $3.66\text{\bf e-}4$& $1.0331$ & $6.16\text{\bf e-}4$ & $0.9740$\\ 
&$0.0125$ & $1.12\text{\bf e-}3$& $1.0506$ & $1.76\text{\bf e-}4$& $1.0562$ & $3.18\text{\bf e-}4$ & $0.9539$ \\
& $0.00625$ & $5.72\text{\bf e-}4$ & $0.9694$ & $8.66\text{\bf e-}5$& $1.0231$ & $1.49\text{\bf e-}4$&$1.0937$  \\
\hline
\end{tabular}
    \caption{Convergence analysis for different step time sizes $h_t$ and two values of $\alpha$ ($\alpha=0.3$ and $\alpha=0.7$) with a fixed mesh size $h_x=0.001$.}
    \label{order_t}
\end{table}

%\begin{figure}[H]
     %\centering
     %\begin{subfigure}[b]{0.45\textwidth}
        % \centering
         %\includegraphics[width=\textwidth]{FIGURES/time-0.3.png}
         %\caption {For $\alpha=0.3$}
         %\label{fig:y equals x}
     %\end{subfigure}
     %\hfill
    % \begin{subfigure}[b]{0.45\textwidth}
         %\centering
         %\includegraphics[width=\textwidth]{FIGURES/time-0.7.png}
         %\caption{For $\alpha=0.7$}
         %\label{fig:three sin x}
    % \end{subfigure}
   
        %\caption{Comparison of Exact and Approximated Solutions for $\alpha=0.5$ with $h_x=h_t=0.001$. The top row shows the exact solutions, while the bottom row illustrates the approximated solutions. }
       % \label{ordert}
%\end{figure}
To this end, we illustrate in Figure \ref{test_direct} the variation of the exact and approximated solutions for $\alpha=0.5$ with $h_x=h_t=0.001.$
\begin{figure}[H]
     \centering
     \begin{subfigure}[b]{0.3\textwidth}
         \centering
         \includegraphics[width=\textwidth]{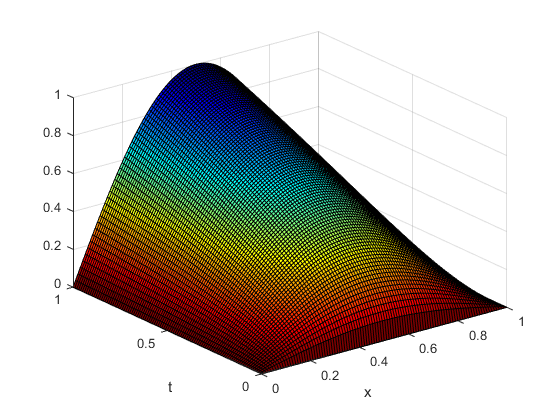}
         \caption {Exact solution $u_{1}^\dagger$}
         \label{fig:y equals x}
     \end{subfigure}
     \hfill
     \begin{subfigure}[b]{0.3\textwidth}
         \centering
         \includegraphics[width=\textwidth]{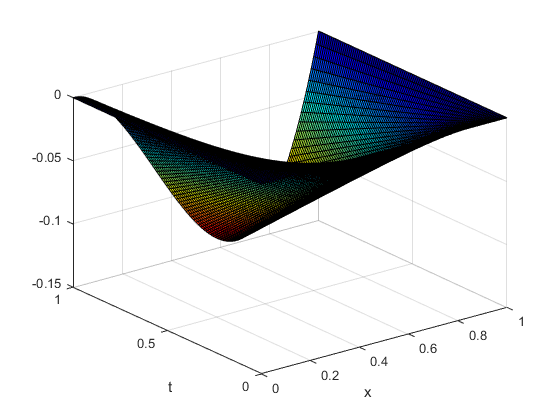}
         \caption{Exact solution $u_{2}^\dagger$}
         \label{fig:three sin x}
     \end{subfigure}
     \hfill
     \begin{subfigure}[b]{0.3\textwidth}
         \centering
         \includegraphics[width=\textwidth]{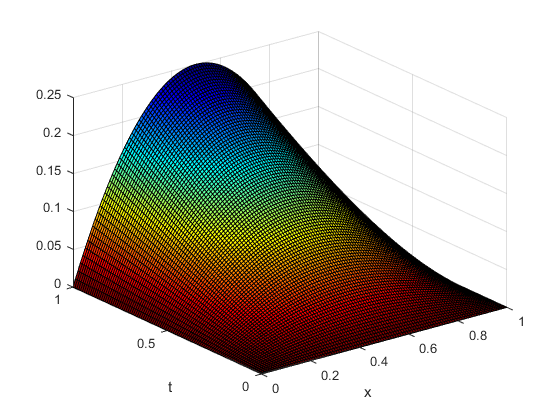}
         \caption{Exact solution $u_{3}^\dagger$}
         \label{fig:five over x}
     \end{subfigure}\\
     \begin{subfigure}[b]{0.3\textwidth}
         \centering
         \includegraphics[width=\textwidth]{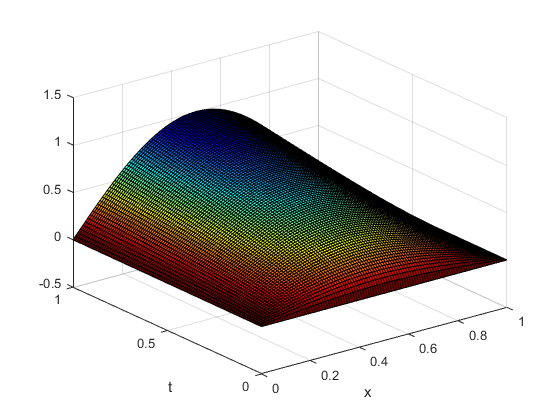}
         \caption{Approximated solution $\tilde{u}_1$}
         \label{fig:y equals x}
     \end{subfigure}
     \hfill
     \begin{subfigure}[b]{0.3\textwidth}
         \centering
         \includegraphics[width=\textwidth]{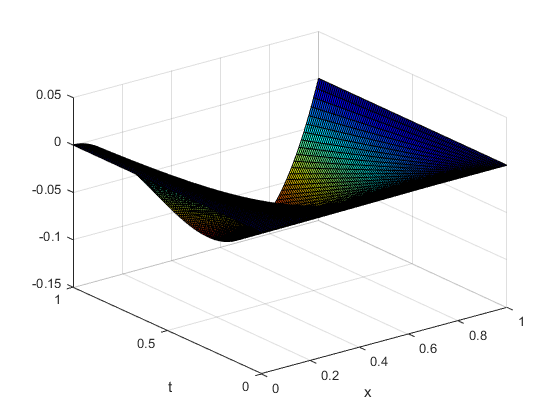}
         \caption{Approximated solution $\tilde{u}_2$}
         \label{fig:three sin x}
     \end{subfigure}
     \hfill
     \begin{subfigure}[b]{0.3\textwidth}
         \centering
         \includegraphics[width=\textwidth]{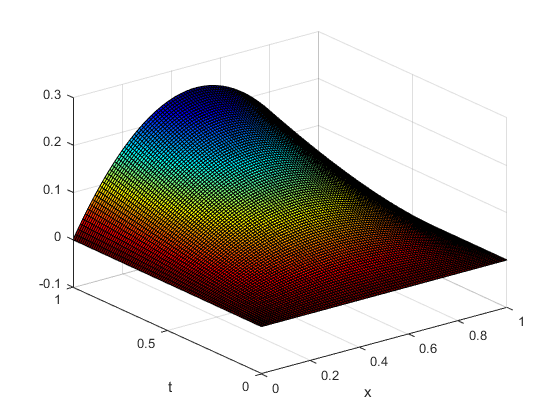}
         \caption{Approximated solution $\tilde{u}_3$}
         \label{fig:five over x}
     \end{subfigure}
        \caption{Comparison of Exact and Approximated Solutions for $\alpha=0.5$ with $h_x=h_t=0.001$. The top row depicts the exact solutions, while the bottom row illustrates the approximated solutions. }
        \label{test_direct}
\end{figure}
\subsection{Numerical solutions of the inverse problem}
In this section, we present numerical results for three test cases to demonstrate the effectiveness of the Iterative Regularizing Ensemble Kalman Method (IREKM). Without loss of generality, we consider the computational domain \(\Omega = (0,1)\) and the time horizon \(T=1\). The direct problem \eqref{eq1} is solved using the proposed method described in the previous subsection. For the finite element algorithm, we discretize both the time and space domains with grid sizes \( h_t = h_x = \frac{1}{100} \).

The source terms for the test cases are given by:
\begin{equation*}
 f_1(x,t) =  f_2(x,t) = f_3(x,t) = 0.   
\end{equation*}
The initial conditions for the variables are:
$$ u_{1,0}(x) = \sin(\pi x), \quad u_{2,0}(x) = x - x^2, \quad u_{3,0}(x) = \sin(2\pi x). $$
To assess the accuracy of the numerical solution, we compute the approximate relative error, defined as  
\[
E_n = \frac{\left\|\overline{D}_n-D^{\dagger}\right\|_2}{\left\|D^{\dagger}\right\|_2},
\]
where \(\overline{D}_n = \left(\overline{D}_{1,n}, \overline{D}_{2,n}, \overline{D}_{3,n}\right)\) represents the mean values at the \(n\)th iteration of IREKM, and \(D^{\dagger} = \left(D_1^{\dagger}, D_2^{\dagger}, D_3^{\dagger}\right)\) is the true solution of the inverse problem. The norm \(\|\cdot\|_2\) denotes the Euclidean norm. Additionally, we monitor the residual \(R_n\) at each iteration, given by  
$$
R_n = \left\|C^{-\frac{1}{2}}\left(d-\bar{z}_n\right)\right\|_2.
$$  
A crucial aspect of iterative algorithms is selecting an appropriate stopping criterion. Here, we employ the stopping rule proposed in \cite{iglesias2015iterative,iglesias2016regularizing,zhang2018bayesian}, selecting \( n_{\star} \) such that  
$$
R_{n_{\star}} \leq \tau \delta,
$$  
where \(\tau > 0\) is a constant, and \(\delta = \left\|C^{-\frac{1}{2}}\left(d-\mathbb{F}\left(D^{\dagger}\right)\right)\right\|_2\) represents the noise level in the observed data.

In our numerical experiments, the prior distributions for \( P(D_1) \), \( P(D_2) \), and \( P(D_3) \) are chosen as  
$$
C_i = K_i A^{-s_i}, \quad i=1,2,3,
$$  
where \( A \) is the Laplace operator with homogeneous Dirichlet boundary conditions, and \( s_i > \frac{1}{2} \) determines the regularity of the Gaussian prior. The initial mean values \(\bar{D}_{i,0}\) for \( i=1,2,3 \) are taken as horizontal lines, with each line depending on the numerical example’s endpoint values. This choice is motivated by the fact that Bayesian inverse problems are highly sensitive to prior information selection. The key parameters used in the experiments are summarized in Table \ref{key}. 
\begin{table}[H]
    \centering
    \begin{tabular}{|c|c|c|c|c|c|c|c|c|c|c|}
\hline
\textbf{Parameter} & $N_e$ & $\vartheta_0$ & $\nu$ & $\tau$ &  $s_1$ & $s_2$ & $s_3$ & $K_1$ & $K_2$ & $K_3$ \\
\hline
\textbf{Value} & 300 & 0.1 & 0.7 & 1.1  & 0.75 & 2 & 2 & 100 & 100 & 100 \\
\hline
    \end{tabular}
    \caption{Numerical values of the parameters used in the reconstruction process.}
    \label{key}
\end{table}

\subsubsection{Reconstruction results}
In the following, we present three test scenarios to validate the proposed reconstruction method for the diffusion terms in an HIV infection model. These scenarios feature different forms of diffusion for infected, healthy, and immune cells, each chosen to reflect various biological processes. The observed data for each test case are synthetic, generated to simulate realistic infection dynamics. It should be noted that all computations are performed using {\it MATLAB} version {\it R2017a}. 
\paragraph{Example 1.} In this test, we apply the proposed algorithm to reconstruct the diffusion terms \(D_1\), \(D_2\), and \(D_3\) in an HIV infection model, where the selection of these coefficients is biologically motivated by the distinct behaviors of infected, healthy, and immune cells. The diffusion of infected cells follows \(D_1(u_1) = 1 + 0.5 u_1\), indicating a linear dependence on their concentration.This means their spread increases proportionally as more cells become infected, representing a moderate yet steady propagation of infection. Healthy cells diffuse according to \(D_2(u_2) = 1 + 0.3 u_2^2\), where the quadratic dependence suggests an adaptive response to tissue damage, allowing faster regeneration as the healthy cell population grows. Immune cells follow \(D_3(u_3) = 1 + 0.2 u_3^3\), with a cubic dependence that captures an aggressive recruitment and activation mechanism, enhancing the immune response as more immune cells accumulate. These choices ensure that infected cells spread at a controlled rate, healthy cells migrate more efficiently as part of tissue repair, and immune cells exhibit a highly dynamic response to infection. This test case validates the reconstruction by employing the aforementioned diffusion terms, providing insight into how accurately they can be determined from observed data and assessing the spatial progression of infection and immune response.\\

The reconstruction results for \(D_1(\cdot)\), \(D_2(\cdot)\), and \(D_3(\cdot)\) in Example 1 are illustrated in Figure \ref{recon-test1} for the parameter value \(\alpha = 0.7\). These results provide insight into the accuracy and effectiveness of the proposed method in recovering the diffusion coefficients, highlighting how well the reconstructed profiles align with the expected behavior.

\begin{figure}[H]
    \centering
    \begin{subfigure}[b]{0.3\textwidth}
        \centering
        \includegraphics[width=\textwidth]{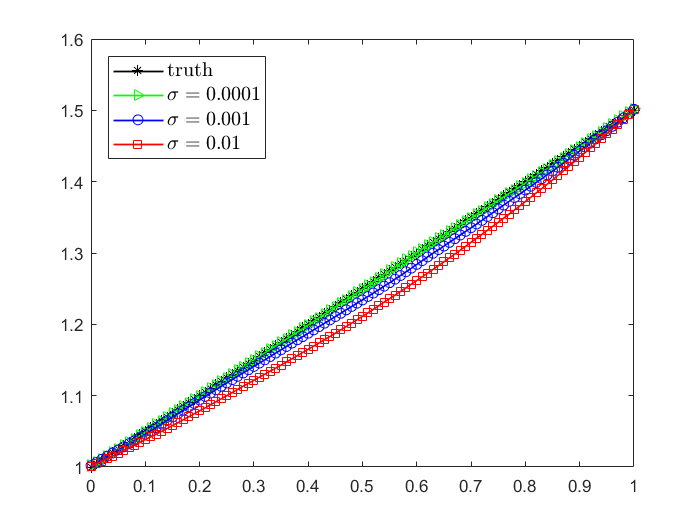}
        \caption{}
    \end{subfigure}
    \hfill
    \begin{subfigure}[b]{0.3\textwidth}
        \centering

        \includegraphics[width=\textwidth]{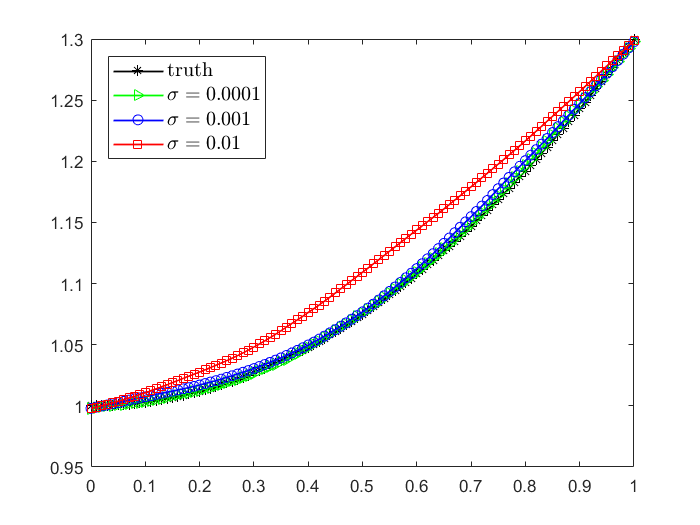}
        \caption{}
    \end{subfigure}
    \hfill
    \begin{subfigure}[b]{0.3\textwidth}
        \centering
        \includegraphics[width=\textwidth]{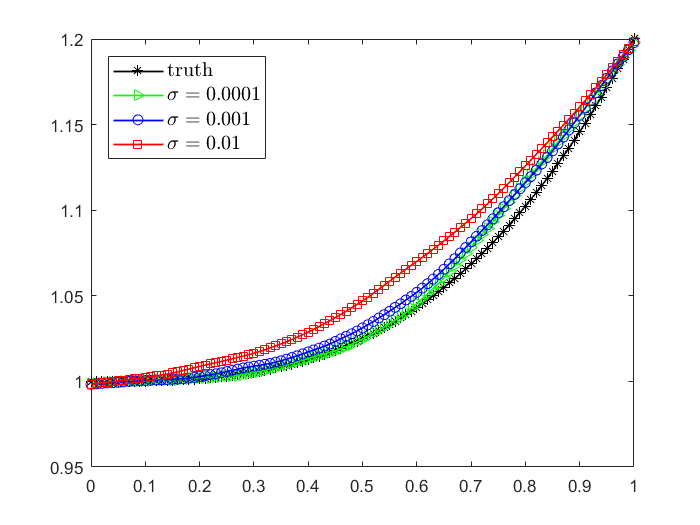}
        \caption{}
    \end{subfigure}
    \caption{The numerical results of $D_1(\cdot)$ (left), $D_2(\cdot)$ (middle) and $D_3(\cdot)$ (right) for different noises $\sigma$ with $\alpha=0.7$ in Example 1.}
    \label{recon-test1}
\end{figure}

\paragraph{Example 2.} Our objective is to reconstruct the diffusion coefficients \(D_1\), \(D_2\), and \(D_3\) in an HIV infection model, where their selection is guided by biological reasoning. The diffusion of infected cells follows \(D_1(u_1) = \frac{1}{1 + u_1^2}\), reflecting a saturation effect that restricts their movement at higher concentrations due to physical barriers or immune responses. Healthy cells diffuse according to \(D_2(u_2) = \frac{1}{1 + u_2^4}\), modeling a scenario in which their mobility decreases as cell density increases, representing spatial competition during tissue repair. Immune cells, governed by \(D_3(u_3) = 1 + \frac{u_3}{1 + u_3^2}\), exhibit a regulated diffusion pattern where their spread initially increases but slows at high concentrations, mimicking immune system limitations such as overcrowding or exhaustion. These diffusion models capture biologically relevant constraints, ensuring a more realistic representation of infection dynamics. This test case evaluates the feasibility of reconstructing these coefficients from the final time value observed data and provides insight into the interplay between infection spread, tissue recovery, and immune response.\\

Figure \ref{recon-test2} shows the reconstruction results for \(D_1(\cdot)\), \(D_2(\cdot)\), and \(D_3(\cdot)\) in Example 2 with the parameter value \(\alpha = 0.5\). These results highlight the method's effectiveness in accurately recovering the diffusion coefficients, with the reconstructed profiles closely matching expected behavior.

\begin{figure}[H]
    \centering
    \begin{subfigure}[b]{0.3\textwidth}
        \centering
        \includegraphics[width=\textwidth]{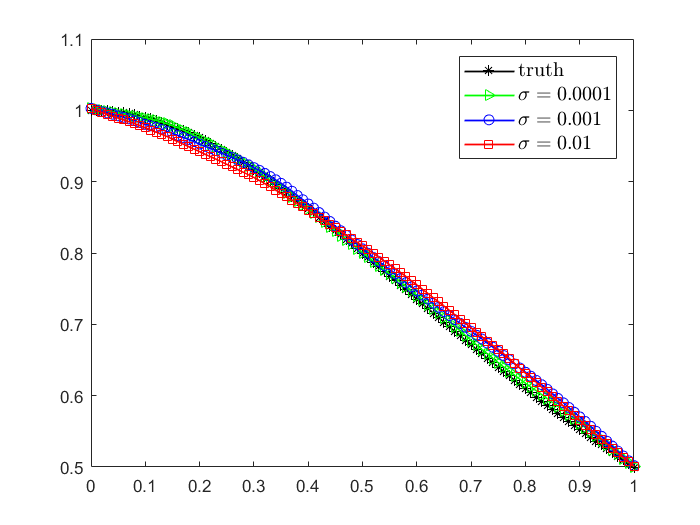}
        \caption{}
    \end{subfigure}
    \hfill
    \begin{subfigure}[b]{0.3\textwidth}
        \centering

        \includegraphics[width=\textwidth]{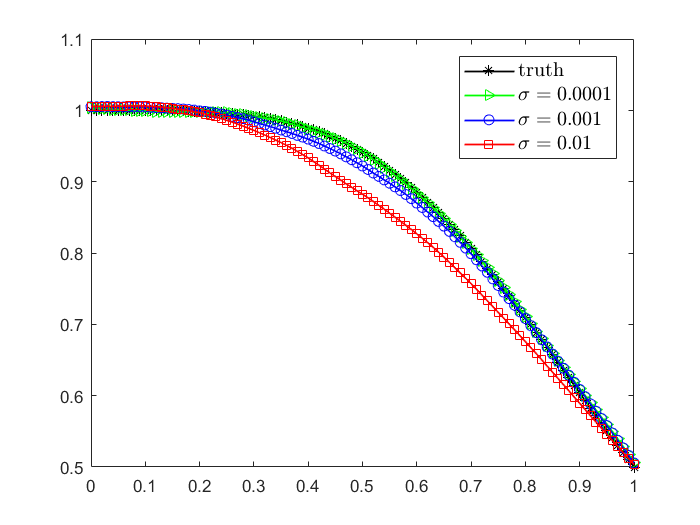}
        \caption{}
    \end{subfigure}
    \hfill
    \begin{subfigure}[b]{0.3\textwidth}
        \centering
        \includegraphics[width=\textwidth]{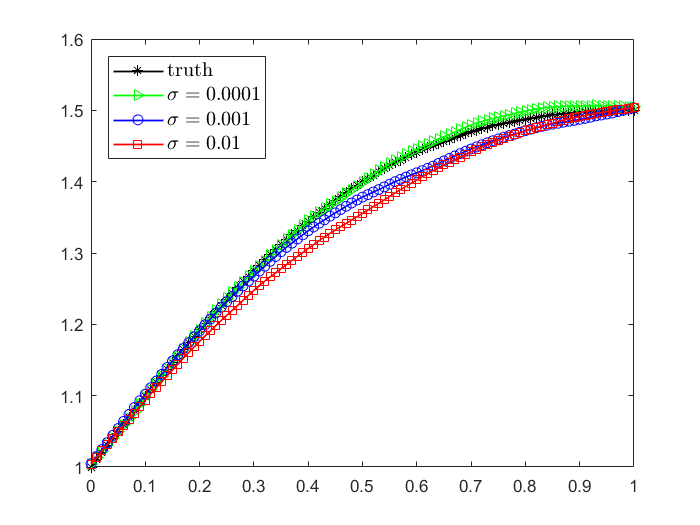}
        \caption{}
    \end{subfigure}
    \caption{The numerical results of $D_1(\cdot)$ (left), $D_2(\cdot)$ (middle) and $D_3(\cdot)$ (right) for different noises $\sigma$ with $\alpha=0.5$ in Example 2.}
    \label{recon-test2}
\end{figure}

\paragraph{Example 3.}   
In this test, we aim to reconstruct the diffusion terms \(D_1\), \(D_2\), and \(D_3\) for an infection model with more complex and nonlinear forms of diffusion. The diffusion of infected cells is governed by \(D_1(u_1) = 1 + \sin(\pi u_1)\), which introduces a sinusoidal variation in the diffusion, reflecting periodic fluctuations in the movement or interaction of infected cells. For healthy cells, the diffusion term is given by \(D_2(u_2) = \exp(u_2^2 - u_2)\), capturing an exponential growth and decay that may represent the tissue’s regenerative or deteriorative response to the presence of healthy cells. The diffusion term for immune cells follows \(D_3(u_3) = 1 + \exp(2u_3)\), reflecting an exponential increase in diffusion as immune cells accumulate, which may correspond to enhanced immune response as the number of activated immune cells grows.\\

This test case is designed to assess the method’s ability to handle nonlinear, complex diffusion terms while providing valuable insight into the model's adaptability to intricate biological processes. The reconstruction results for \(D_1(\cdot)\), \(D_2(\cdot)\), and \(D_3(\cdot)\) in Example 3 are shown in Figure \ref{recon-test3} for the parameter value \(\alpha = 0.2\). These results demonstrate the method's capability to accurately recover the complex forms of the diffusion terms, illustrating how the reconstructed profiles align with expected behaviors of infection dynamics, tissue repair, and immune response.
\begin{figure}[H]
    \centering
    \begin{subfigure}[b]{0.3\textwidth}
        \centering
        \includegraphics[width=\textwidth]{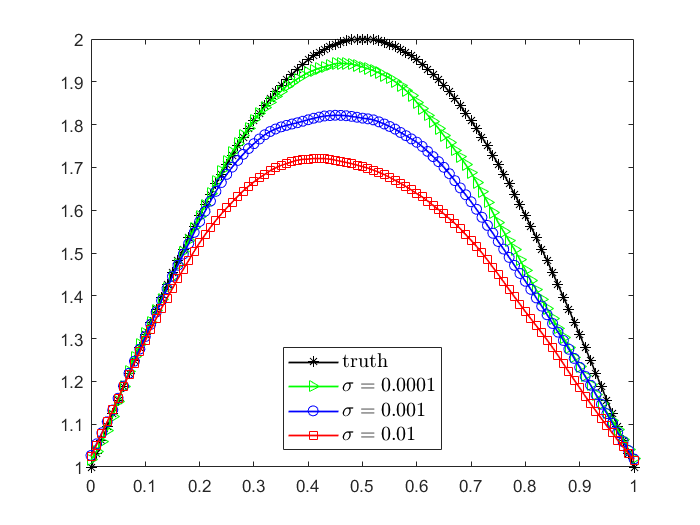}
        \caption{}
    \end{subfigure}
    \hfill
    \begin{subfigure}[b]{0.3\textwidth}
        \centering

        \includegraphics[width=\textwidth]{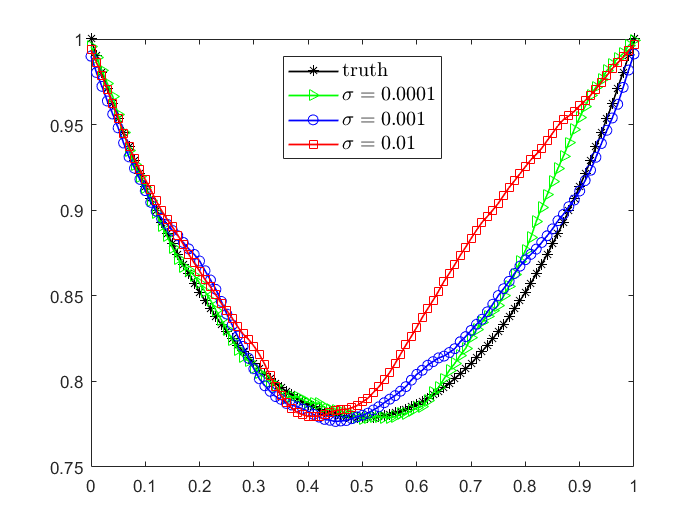}
        \caption{}
    \end{subfigure}
    \hfill
    \begin{subfigure}[b]{0.3\textwidth}
        \centering
        \includegraphics[width=\textwidth]{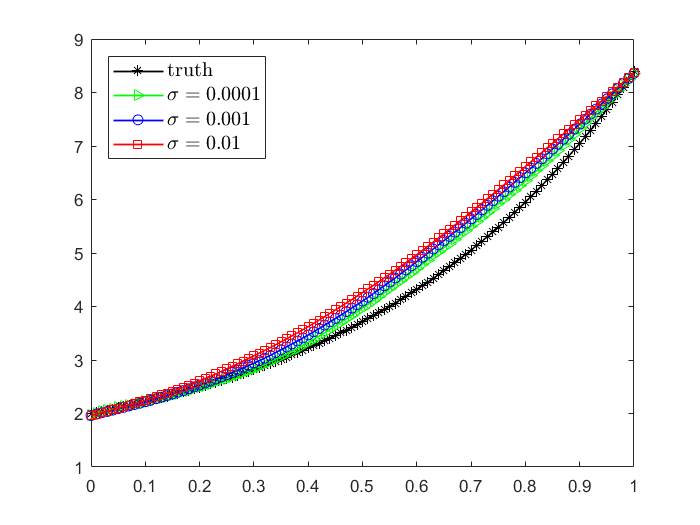}
        \caption{}
    \end{subfigure}
    \caption{The numerical results of $D_1(\cdot)$ (left), $D_2(\cdot)$ (middle) and $D_3(\cdot)$ (right) for different noises $\sigma$ with $\alpha=0.2$ in Example 3.}
    \label{recon-test3}
\end{figure}

\subsubsection{Stability and convergence}
The results presented in Figures \ref{recon-test1}, \ref{recon-test2}, and \ref{recon-test3} demonstrate that reconstruction accuracy and precision improve as noise levels decrease. As noise is reduced, the numerical solution increasingly aligns with the true solution, highlighting the effectiveness of the method in achieving higher precision with cleaner data. This emphasizes IREKM's sensitivity to noise and its ability to achieve greater accuracy in low-noise conditions. Furthermore, our approach successfully reconstructs various functional forms, including linear, quadratic, exponential, logistic, and sinusoidal behaviors. To further illustrate the stability and convergence of the proposed IREKM algorithm in addressing our inverse problem, Figure \ref{performance} depicts the logarithmic relative errors \(\log_{10}(E_n)\) and residuals \(\log_{10}(R_n)\) as functions of iteration count \(n\) across different noise levels.  

\begin{figure}[H]
    \centering
    \begin{subfigure}[b]{0.3\textwidth}
        \centering
\includegraphics[width=\textwidth]{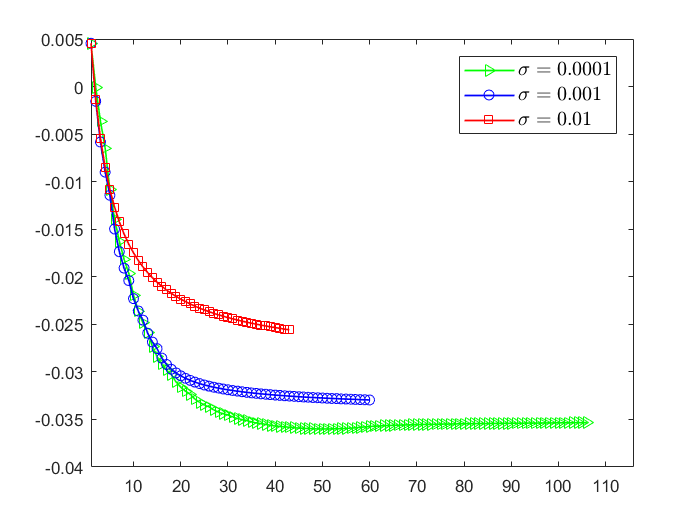}
        \caption{}
    \end{subfigure}
    \hfill
    \begin{subfigure}[b]{0.3\textwidth}
        \centering
\includegraphics[width=\textwidth]{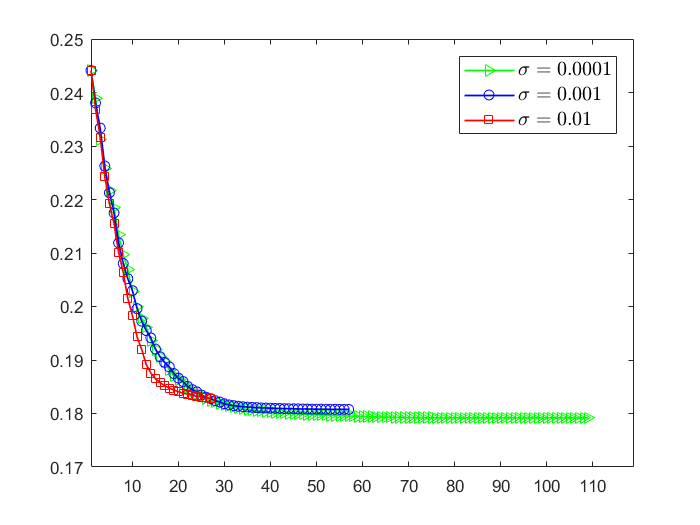}
        \caption{}
    \end{subfigure}
    \hfill
    \begin{subfigure}[b]{0.3\textwidth}
        \centering
\includegraphics[width=\textwidth]{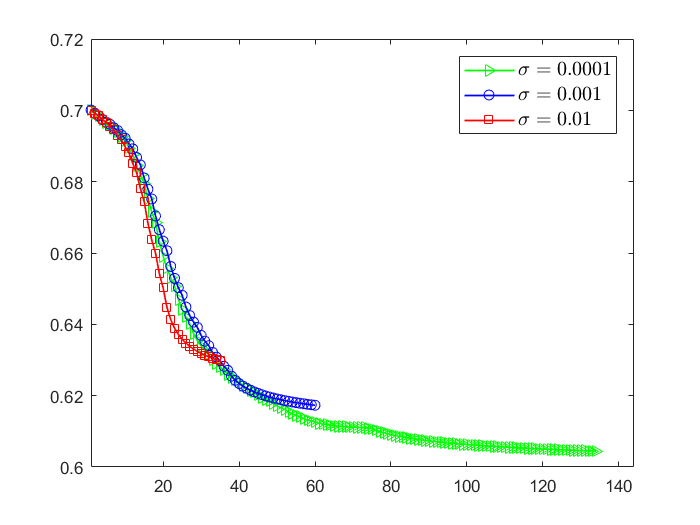}
        \caption{}
    \end{subfigure}
    \hfill
      \begin{subfigure}[b]{0.3\textwidth}
        \centering
\includegraphics[width=\textwidth]{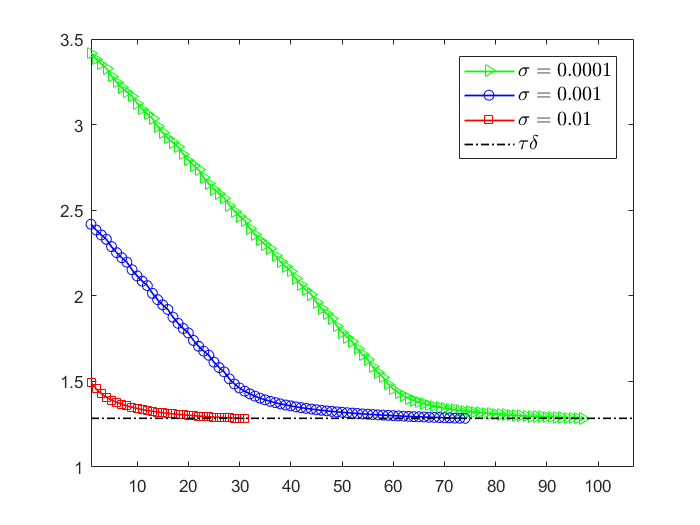}
        \caption{}
    \end{subfigure}
    \hfill
    \begin{subfigure}[b]{0.3\textwidth}
        \centering
\includegraphics[width=\textwidth]{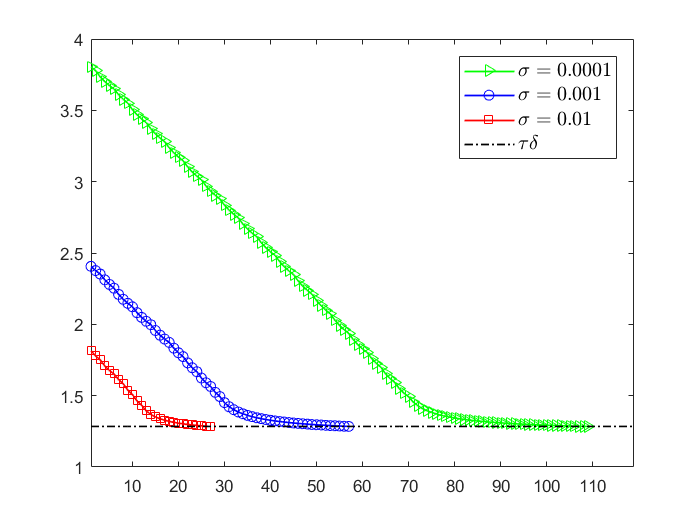}
        \caption{}
    \end{subfigure}
    \hfill
    \begin{subfigure}[b]{0.3\textwidth}
        \centering
\includegraphics[width=\textwidth]{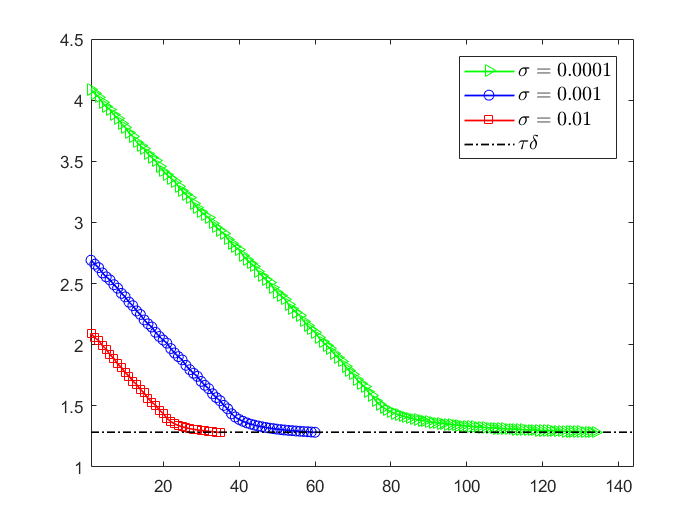}
        \caption{}
    \end{subfigure}
 \caption{Logarithmic relative errors (\(E_n\), top row) and logarithmic residuals (\(R_n\), bottom row) for different noise levels: \(\sigma = 0.0001\), \(\sigma = 0.001\), and \(\sigma = 0.01\). From left to right: results for Example 1, Example 2  and Example 3.}
    \label{performance}
\end{figure}

Figure \ref{performance} illustrates the stability  and convergence properties of the Iterative Regularizing Ensemble Kalman Method (IREKM) when applied to the inverse problem. As noise decreases, the approximation errors progressively diminish, demonstrating the method's robustness. The stopping criterion, based on the residual \(R_n\),  ensures proper algorithm termination by enforcing the condition  \(R_n \leq \tau \delta\), where \(\delta\) quantifies the noise level. This avoids overfitting and ensures convergence. The plotted residuals confirm the correct implementation of the stopping criterion, as the errors consistently decrease throughout the iterations. Additionally, the figures demonstrate that both the error and residual stabilize after a few iterations, indicating rapid convergence to an acceptable solution. Notably, when noise levels are lower, the algorithm requires more iterations to meet the stopping criterion. This is because achieving higher precision in low-noise scenarios demands further refinement of the solution, requiring additional computational steps to reduce the residual to the desired threshold.\\

In conclusion, the Iterative Regularizing Ensemble Kalman Method (IREKM) proves to be a robust and efficient approach for solving the inverse problem of identifying diffusion terms in the HIV infection model. The method exhibits strong stability and convergence properties, as approximation errors decrease with lower noise levels. Furthermore, IREKM shows notable sensitivity to noise, yielding precise reconstructions, particularly in low-noise scenarios.

\section{ Concluding Remarks}  \label{remarks}

In this study, we investigated an inverse problem associated with a time-fractional HIV infection model incorporating nonlinear diffusion. Our primary objective was to recover unknown diffusion functions using final time measurement data. Given the ill-posedness of the inverse problem and the presence of measurement noise, we employed a Bayesian inference framework combined with the Iterative Regularizing Ensemble Kalman Method (IREKM) to achieve stable and reliable reconstructions.  The numerical experiments demonstrated the effectiveness of the proposed methodology in reconstructing the unknown diffusion terms under various noise levels. The results indicated that the IREKM method successfully captured the complex diffusion dynamics of the HIV infection model while maintaining robustness against observational noise. Furthermore, the Bayesian approach allowed for uncertainty quantification, providing a comprehensive assessment of the estimated parameters.  

Key findings from this research include:
\begin{enumerate}[label=(\roman*)]
\item The proposed framework is capable of accurately recovering nonlinear diffusion terms in the HIV infection model, even under moderate noise conditions.
\item  The stability and convergence of the method were validated through multiple test cases, showing improved performance as noise levels decreased.
\item The sensitivity analysis confirmed that IREKM is effective in handling inverse problems for nonlocal and nonlinear models, making it a promising tool for studying complex biological systems.  
\end{enumerate}

While the proposed approach demonstrates strong performance, several directions for future research remain. Extending the method to higher-dimensional problems and incorporating real experimental data could offer additional insights into HIV infection dynamics. Furthermore, improving computational efficiency, particularly for large-scale simulations, remains an important area for development. Lastly, exploring alternative regularization techniques may further enhance the robustness of the inverse problem formulation. Overall, this study contributes to the growing field of inverse problems in fractional differential models and provides a computational framework to advance the understanding of HIV infection mechanisms. 

\end{document}